%% file: MMSLinearNov.tex
\documentclass[12pt]{article}

\usepackage{a4wide}
\usepackage{graphicx}
\usepackage{amsmath}
\usepackage{amssymb}
\usepackage{amsthm}
\usepackage{enumerate}
\usepackage{xspace}

\newtheorem{theorem}{Theorem}[section]
\newtheorem{definition}[theorem]{Definition}

\newtheorem{lemma}[theorem]{Lemma}
\newtheorem{claim}[theorem]{Claim}

\newtheorem{conjecture}[theorem]{Conjecture}

\newtheorem{problem}[theorem]{Problem}
\usepackage[margin=20pt,small,bf]{caption}

\newcommand{\body}{paper\xspace}

\newcommand{\Hyper}[1]{\mathcal{#1}}
\newcommand{\ILF}[2]{{#1}/{#2}}
\newcommand{\CompHyper}[2]{\Hyper{K}_{#1} ^{(#2)}}

\begin{document}

\title{A linear bound on the Manickam-Mikl\'os-Singhi Conjecture}

\author{\large{Alexey Pokrovskiy} \thanks{Research supported by the LSE postgraduate research studentship scheme.}
\\
\\Department of Mathematics,
\\ {London School of Economics and Political Sciences,} 
\\ London WC2A 2AE, UK 
\\ {Email: \texttt{a.pokrovskiy@lse.ac.uk}}}

\maketitle

\begin{abstract}
Suppose that we have a set of numbers $x_1, \dots, x_n$ which have nonnegative sum.  How many subsets of $k$ numbers from $\{x_1, \dots, x_n\}$ must have nonnegative sum?  Manickam, Mikl\'os, and Singhi conjectured that for $n\geq 4k$ the answer is $\binom{n-1} {k-1}$.  This conjecture is known to hold when $n$ is large compared to $k$.  The best known bounds are due to Alon, Huang, and Sudakov who proved the conjecture when $n\geq 33k^2$.  In this paper we improve this bound by showing that there is a constant $C$ such that the conjecture holds when $n\geq C k$ .  
\end{abstract}

\section{Introduction}
\input{MMSIntroduction}

\section{Proof of the averaging lemma}\label{MMSAveragingSection}
\input{MMSAveraging}

\section{Construction of the hypergraphs $\Hyper{H}_{n,k}$}
\input{MMSConstruction} \label{MMSConstructionSection}

\section{Hypergraphs of order $O(k^4)$ with the MMS-property}
\input{MMSProofSimple}\label{MMSSimpleProofSection}

\section{Proof of Theorem~\ref{MMSMainTheorem}}\label{MMSProofSection}
\input{MMSProofNew}

\section{Remarks}\label{MMSRemarksSection}
\input{MMSDiscussion}


\bigskip\noindent
\textbf{Acknowledgment}

\smallskip\noindent
The author would like to thank  Peter Allen, Jan van den Heuvel,  Jozef Skokan, and Benny Sudakov for their advice and discussions.

\bibliography{pathpartition}
\bibliographystyle{abbrv}
\end{document}

%% file: MMSIntroduction.tex
Suppose that we have a set of numbers $x_1, \dots, x_n$ satisfying $x_1+ \dots+ x_n\geq 0$.  How many subsets $A\subset \{x_1, \dots, x_n\}$ must satisfy $\sum_{a\in A} a\geq 0$?

By choosing $x_1=n-1$ and $x_2=\dots=x_n=-1$ we see that the answer to this question can be at most $2^{n-1}$.
In fact, this example has the minimal number of nonnegative sets.  Indeed, for any set $A\subset \{x_1, \dots, x_n\}$ either $A$ or $\{x_1, \dots, x_n\}\setminus A$ must have nonnegative sum, so there must always be at least $2^{n-1}$ nonnegative subsets in any set of numbers $\{x_1, \dots, x_n\}$ with nonnegative sum.  

A more difficult  problem arises if we count only subsets of fixed order.  By again considering the example when $x_1=n-1$ and $x_2=\dots=x_n=-1$ we see that there are sets of $n$ numbers with nonnegative sums which have only $\binom{n-1} {k-1}$ nonnegative $k$-sums (sums of $k$ distinct numbers).
Manickam, Mikl\'os, and Singhi conjectured that for $n\geq 4k$ this assignment gives the least possible number of nonnegative $k$-sums.
\begin{conjecture}  [Manickam, Mikl\'os, Singhi, \cite{Manickam, Singhi}] \label{MMSConjecture}
Suppose that $n\geq 4k$, and we have $n$ real numbers $x_1, \dots, x_n$ such that $x_1 + \dots + x_n \geq 0$. Then, at least $\binom{n-1}  {k-1}$ subsets $A\subset \{x_1, \dots, x_n\}$ of order $k$ satisfy $\sum_{a\in A} a\geq 0$
\end{conjecture}
Conjecture~\ref{MMSConjecture} appeared in~\cite{Singhi} where it was phrased in terms of calculating invariants of an association scheme known as the \emph{Johnson Scheme}.  In~\cite{Manickam}, Conjecture~\ref{MMSConjecture} was phrased in the combinatorial form in which it is stated above.

A motivation for the bound ``$n\geq 4k$'' is that for $k\geq 3$ and  $n=3k+1$ there exists an assignment of values to $x_1, \dots, x_{3k+1}$ which results in less than $\binom{n-1} {k-1}$ nonnegative $k$-sums.  Indeed, letting $x_1=x_2=x_3=2-3k$ and $x_4=\dots=x_{3k+1}=3$ gives an assigment satisfying $x_1+\dots+x_{3k+1}=0$ but having ${\binom{3k-2}{k}}$ nonnegative $k$-sums, which is less than ${\binom{3k} {k-1}}$ for $k\geq 3$.

Conjecture~\ref{MMSConjecture} has been open for over two decades, and many partial results have been proven.  The conjecture has been proven for $k\leq 3$ by Manickam \cite{ManickamSmalln} and independently by Chiaselotti and Marino \cite{Marino}.  It has been proven whenever $n\equiv 0 \pmod{k}$ by Manickam and Singhi \cite{Singhi}.
In addition several results have been proved establishing the conjecture when $n$ is large compared to $k$.  Manickam and Mikl\'os \cite{Manickam} showed that the conjecture holds when $n\geq (k - 1)(k^k + k^2 ) + k$ holds.  Tyomkyn~\cite{Tyomkyn} improved this bound to  $n\geq k(4e \log k)^k \approx e^{ck\log \log k}$.  Recently Alon, Huang, and Sudakov \cite{Sudakov} showed that the conjecture holds when $n\geq 33k^2$.  
The aim of this \body     is to improve these bounds by showing that the conjecture holds in a range when $n$ is linear with respect to $k$.

\begin{theorem} \label{LinearMMS}
Suppose that $n\geq 10^{46}k$, and we have $n$ real numbers $x_1, \dots, x_n$ such that $x_1 + \dots + x_n \geq 0$. At least $\binom{n-1}  {k-1}$ subsets $A\subset \{x_1, \dots, x_n\}$ of order $k$ satisfy $\sum_{a\in A} a\geq 0$
\end{theorem}

It is worth noticing at this point that there seem to be connections between the problem and results mentioned so far in this \body, and the Erd\H{o}s-Ko-Rado Theorem about intersecting families of sets.  A family $\mathcal{A}$ of sets is said to be \emph{intersecting} if any two members of $\mathcal{A}$ intersect.
The Erd\H{o}s-Ko-Rado Theorem \cite{ErdosKoRado} says that for $n\geq 2k$, any intersecting family $\mathcal{A}$ of subsets of $[n]$ of order $k$, must satisfy $|\mathcal{A}|\leq {\binom{n-1}{k-1}}$.  The extremal family of sets in the Erd\H{o}s-Ko-Rado Theorem is formed by considering the family of all $k$-sets which contain a particular element of $[n]$.  This is exactly the family $\mathcal{A}$ that we obtain from the extremal case of the Manickam-Mikl\'os-Singhi Conjecture if we let the members of $\mathcal{A}$ be the nonnegative $k$-sums from $x_1, \dots, x_n$. In addition, many of the methods used to approach Conjecture~\ref{MMSConjecture} are similar to proofs of the Erd\H{o}s-Ko-Rado Theorem.  
The method we use to prove Theorem~\ref{LinearMMS} in this \body is inspired by Katona's proof of the Erd\H{o}s-Ko-Rado Theorem in \cite{Katona}.

Suppose that we have a hypergraph ${\Hyper{H}}$ together with an assignment of real numbers to the vertices of ${\Hyper{H}}$ given by $f:V({\Hyper{H}})\to \mathbb{R}$. We can extend $f$ to the powerset of $V({\Hyper{H}})$ by letting $f(A)=\sum_{v\in A} f(v)$ for every $A\subseteq V({\Hyper{H}})$.  We say that an edge $e \in E({\Hyper{H}})$ is \emph{negative} if $f(e)<0$, and $e$ is \emph{nonnegative} otherwise. We let $e_f ^+({\Hyper{H}})$ be the number of nonnegative edges of ${\Hyper{H}}$.  
Recall that the degree $d(v)$ of a vertex $v$ in a hypergraph $\Hyper{H}$ is the number of edges containing $v$.  A hypergraph $\Hyper{H}$ is $d$-regular if every vertex has degree $d$.   The minimum degree of a hypergraph $\Hyper{H}$ is $\delta(\Hyper{H})=\min_{v\in V(\Hyper{H})} d(v)$.  The $k$-uniform complete hypergraph on $n$ vertices is denoted by $\CompHyper{n}{k}$.

The following observation is key to our proof of Theorem~\ref{LinearMMS}.
\begin{lemma} \label{MMSAveraging}
Let ${\Hyper{H}}$ be a $d$-regular $k$-uniform hypergraph on $n$ vertices. Suppose that for every $f:V(\Hyper{H})\to \mathbb{R}$ satisfying $\sum_{x \in V(\Hyper{H})} f(x)\geq 0$ we have $e_f ^+({\Hyper{H}})\geq d$. Then for every $f:V(\CompHyper{n}{k})\to \mathbb{R}$ satisfying $\sum_{x \in V(\CompHyper{n}{k})} f(x)\geq 0$ we have $e_f ^+(\CompHyper{n}{k})\geq \binom{n-1}{k-1}$ (and so Conjecture~\ref{MMSConjecture} holds for this particular $n$ and $k$).
\end{lemma}
Lemma~\ref{MMSAveraging} is proved by an averaging technique similar to Katona's proof of the Erd\H{o}s-Ko-Rado Theorem (see Section~\ref{MMSAveragingSection}).  This technique has already appeared in the context of the Manickam-Mikl\'os-Singhi Conjecture in \cite{Manickam} where it was used to prove the conjecture when $n\geq (k - 1)(k^k + k^2 ) + k$.

Lemma~\ref{MMSAveraging} shows that instead of proving the conjecture about the complete graph $\CompHyper{n}{k}$, it may be possible to find regular hypergraphs which satisfy the condition in Lemma~\ref{MMSAveraging} and hence deduce the conjecture. 
This motivates us to make the following definition.
\begin{definition}
A $k$-uniform hypergraph $\Hyper{H}$ has the \textbf{MMS-property} if for every $f:V(\Hyper{H})\to \mathbb{R}$ satisfying $\sum_{x \in V(\Hyper{H})} f(x)\geq 0$ we have $e^+({\Hyper{H}})\geq \delta(\Hyper{H})$.  
\end{definition}
Conjecture~\ref{MMSConjecture} is equivalent to the statement that for $n\geq 4k$ the complete hypergraph on $n$ vertices has the MMS-property.
Lemma~\ref{MMSAveraging} shows that in order to prove Conjecture~\ref{MMSConjecture} for particular $n$ and $k$, it is sufficient to find one regular $n$-vertex $k$-uniform hypergraph $\Hyper{H}$ with the MMS-property.  This hypergraph $\Hyper{H}$ may be much sparser than the complete hypergraph---allowing for very different proof techniques.

Perhaps the first two candidates one chooses for hypergraphs that may have the MMS-property are matchings and tight cycles.
The matching $\Hyper{M}_{t,k}$ is defined as the $k$-uniform hypergraph consisting of $tk$ vertices and $t$ vertex disjoint edges.  Notice that $\Hyper{M}_{t,k}$ is 1-regular.  The matching $\Hyper{M}_{t,k}$ always has the MMS-property---indeed we have that $\sum_{e \in E(\Hyper{M}_{t,k})}f(e)= \sum_{x \in \Hyper{M}_{t,k}} f(x)\geq 0$, and so one of the edges of $\Hyper{M}_{t,k}$ is nonnegative.  This observation was used in \cite{Singhi} to prove Conjecture~\ref{MMSConjecture} whenever $k$ divides $n$.

The tight cycle $\Hyper{C}_{n,k}$ is defined as the hypergraph with vertex set $\mathbb{Z}_n$ and edges formed by the intervals $\{i\pmod{n}, i+1\pmod{n}, \dots, i+k \pmod{n}\}$ for $i\in\mathbb{Z}_n$.  It turns out that the tight cycles do not have the MMS-property when $n\not\equiv 0\pmod{k}$.  To see this for example when $k=3$ and $n\equiv 1 \pmod k$, let $f(x) = 50,$ $50,$ $50,$ $-101,$ $50,$ $50,$ $-101,$ $50,$ $50,$ $-101 \dots $ for $x = 1,2,3,4,5,6,7,8,9,10, \dots$.

An interesting question, which we will return to in Section~\ref{MMSRemarksSection} is ``which hypergraphs have the MMS-property?''

The main result of this \body is showing that there exist $k(k-1)^2$-regular $k$-uniform hypergraphs on $n$ vertices which have the MMS-property, for all $n\geq 10^{46}k$.
\begin{theorem}\label{MMSMainTheorem}
For $n\geq 10^{46}k$, there are $k(k-1)^2$-regular $k$-uniform hypergraphs on $n$ vertices, $\Hyper{H}_{n,k}$, with the property that for every $f:V(\Hyper{H}_{n,k})\to \mathbb{R}$ satisfying $\sum_{x \in V(\Hyper{H}_{n,k})} f(x)\geq 0$ we have $e^+(\Hyper{H}_{n,k})\geq k(k-1)^2$.
\end{theorem}

Combining Theorem~\ref{MMSMainTheorem} and Lemma~\ref{MMSAveraging} immediately implies Theorem~\ref{LinearMMS}.

Throughout this \body, we will use notation from Additive Combinatorics for sumsets $A + B=\{a+b: a\in A, b\in B\}$ and translates $A + x=\{a+x:a\in A\}$. 
For all standard notation we refer to~\cite{CombinatoricsBook}.

The structure of this \body is as follows.  In Section~\ref{MMSAveragingSection} we prove Lemma~\ref{MMSAveraging}.  In Section~\ref{MMSConstructionSection}, we define the graphs $\Hyper{H}_{n,k}$ used in Theorem~\ref{MMSMainTheorem} and prove some of their basic properties.  In Section~\ref{MMSSimpleProofSection}, we prove Theorem~\ref{MMSMainTheorem} with the weaker bound of $n\geq 14k^4$ in order to illustrate the main ideas in the proof of Theorem~\ref{MMSMainTheorem}.  In Section~\ref{MMSProofSection} we prove Theorem~\ref{MMSMainTheorem}.  In Section~\ref{MMSRemarksSection}, we conclude by discussing the techniques used in this \body and whether they could be used to prove Conjecture~\ref{MMSConjecture} in general.

%% file: MMSAveraging.tex
Here we prove Lemma~\ref{MMSAveraging}.
\begin{proof}
Suppose that we have a function $f:\{1, \dots, n\}\to \mathbb{R}$ satisfying $\sum_{x \in \{1, \dots, n\}} f(x)\geq 0$.
Consider a random permutation $\sigma$ of $\{1, \dots, n\}$, chosen uniformly out of all permutations of $\{1, \dots, n\}$. We define a function $f_{\sigma}:\{1, \dots, n\} \to \mathbb{R}$ given by $f_{\sigma}:x\to f(\sigma(x))$. Clearly $\sum_{x \in \{1, \dots, n\}} f_{\sigma}(x)\geq 0$.   We will count $\mathbb{E}(e^+_{f_{\sigma}}({\Hyper{H}}))$ in two different ways.
For an edge $e\in \CompHyper{n}{k}$, we have
$$\mathbb{P}(\sigma(e)\in \Hyper{H}) =\frac{e(\Hyper{H})}{\binom{n}{k}} = \frac{d}{\binom{n-1}{k-1}}$$
Therefore we have
$$ \mathbb{E}(e^+_{f_{\sigma}}({\Hyper{H}})) = \sum_{\substack{e\in \CompHyper{n}{k}, \\f(e)\geq 0}} \mathbb{P}(\sigma(e)\in \Hyper{H}) = e^+(\CompHyper{n,k}) \frac{d}{\binom{n-1}{k-1}}$$
However, by the assumption of the lemma, $\mathbb{E}(e^+_{f_{\sigma}}({\Hyper{H}}))$ is at least $d$. This gives us
$$e^+(\CompHyper{n}{k})\geq \binom{n-1}{k-1}. $$
\end{proof}

%% file: MMSConstruction.tex
In this section we construct graphs $\Hyper{H}_{n,k}$ which satisfy Theorem~\ref{MMSMainTheorem}.  We also prove some basic properties which the graphs $\Hyper{H}_{n,k}$ have.

Define the clockwise interval between $a$ and $b \in \mathbb{Z}_n$ to be  $[a,b]=\{a, a+1, \dots ,b\}$.
The graph $\Hyper{H}_{n,k}$ has vertex set $\mathbb{Z}_n$.  We define $k$-edges $e(v,i,j)$ as follows:

$$ e(v,i,j)= [v,v+i-1]\cup [v+i+j,  v + j + k -1]$$

The edges of $\Hyper{H}_{n,k}$ are given by $e(v,i,j)$ for $v\in \mathbb{Z}_n$ and $i,j \in \{1,\dots,k-1\}$.  In other words $\Hyper{H}_{n,k}$ consists of all the double intervals of order $k$, where the distance between the two intervals is at most $k-1$.

Notice that the graph $\Hyper{H}_{n,k}$ is indeed $k(k-1)^2$ regular. 

In order to deal with the graphs $\Hyper{H}_{n,k}$ it will be convenient to assign a particular set $E(v)$ of $O(k^2)$ edges to each vertex $v$.
First, for each vertex $v$ in $\Hyper{H}_{n,k}$ and $i,j \in [1,k-1]$,  we will define a set of edges, $E(v,i,j)$.  Then $E(v)$ will be a union of the sets $E(v,i,j)$.

The definition of the sets $E(v,i,j)$ is quite tedious.  However the sets $E(v,i,j)$ are constructed to satisfy only a few properties.  One property that we will need is that for fixed, $v$, $i$, $j$  certain intervals can be formed as disjoint unions of edges in $E(v,i,j)$.  See Figures~\ref{fig:evig1} -- \ref{fig:evig4} for illustrations of the precise configurations that we will use.  Another property that we will need is that no edge $e \in \Hyper{H}_{n,k}$ is contained in too many of the sets $E(v,i,j)$.  See Lemmas~\ref{nonnegativeintervals} and~\ref{spreadout} for precise statements of these two properties.

Over the next four pages we define the sets $E(v,i,j)$.  

\newpage

If $i+j \geq k$ and $i\geq j$, then we let
\begin{align*}
E(v,i,j)=\{&e(v,i,j), e(v+k+j,i,i+j-k), \\
&e(v+k+i+j, i+j-k, 2k-2i),  e(v+i, j, k-i),\\
&e(v+k+i+2j, k-i,2k-i-j), e(v+i, j, 2k-i-j), \\
&e(v+3k-j, i, j), e(v+3k - j+i, j, k-i),   \\    
&e(v+i, i+j-k, 2k-2i), e(v+i+j, k-i, 2k-i-j), \\
&e(v+2k,i,j), e(v+2k+i,j,k-i)\}.
\end{align*}                 
\begin{figure}
  \centering
    \includegraphics[width=0.87\textwidth]{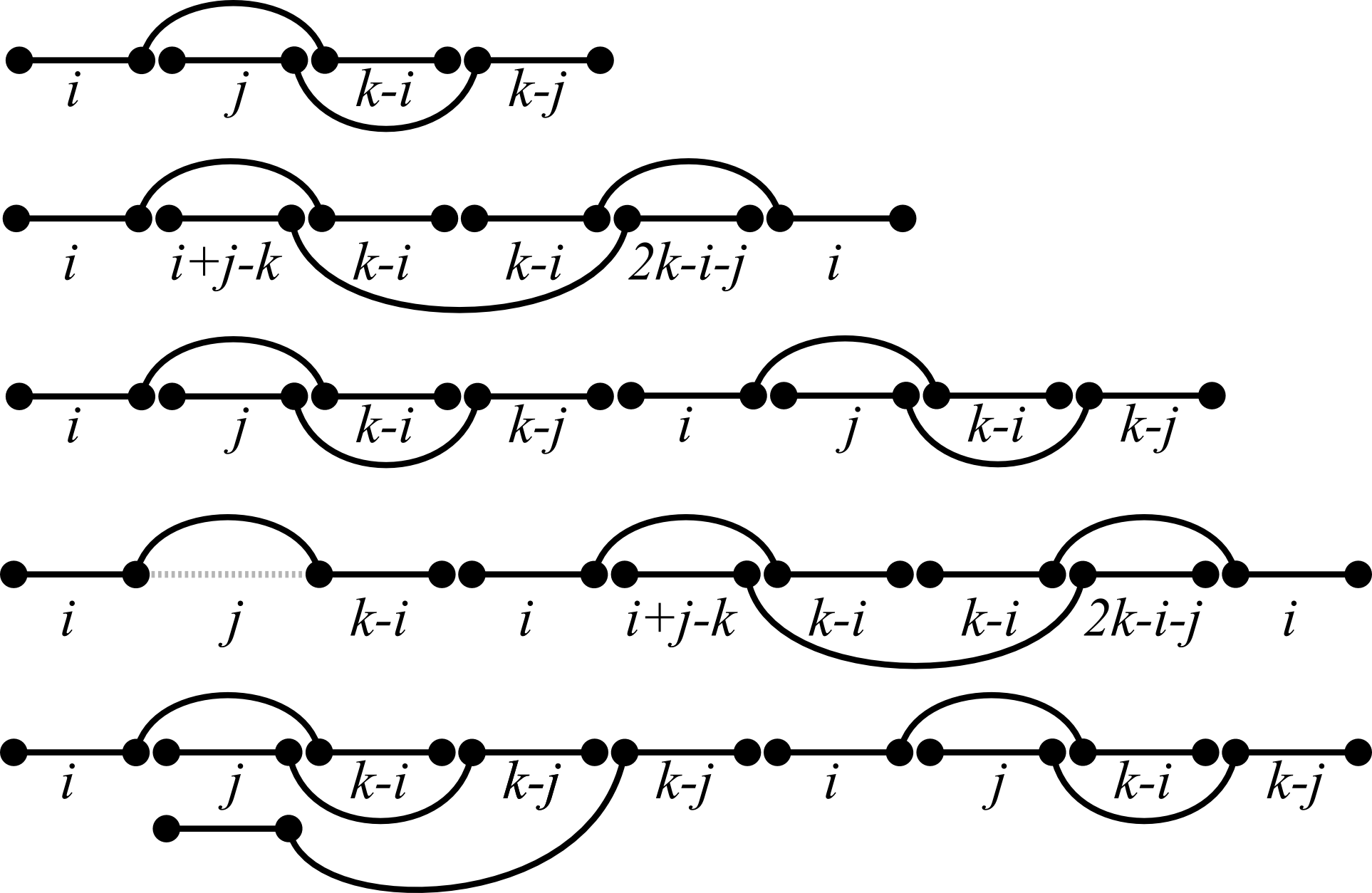}
  \caption{The edges in $E(v,i,j)$  when we have $i+j \geq k$ and $i\geq j$.} \label{fig:evig1}
\end{figure}               
  \newpage
If $i+j \geq k$ and $j< i$, then we let
\begin{align*}
E(v,i,j)=\{&e(v,i,j), e(v+k+j,j,i+j-k), \\ 
&e(v+k+2j, i+j-k, 2k-2j), e(v+i, j, k-i), \\
&e(v+k+i+2j, k-j,2k-i-j),\\
                 &e(v+i, j, 2k-i-j), e(v+3k-j, i, j), \\
&e(v+3k - j+i, j, k-i),       e(v, j, i+j-k),\\ 
&e(v+j, i+j-k, 2k-2j), e(v+i+j, k-j, 2k-i-j),\\ 
&e(v+2k,i,j), e(v+2k+i,j,k-i)
                \}.
\end{align*}   
\begin{figure}
  \centering
    \includegraphics[width=0.87\textwidth]{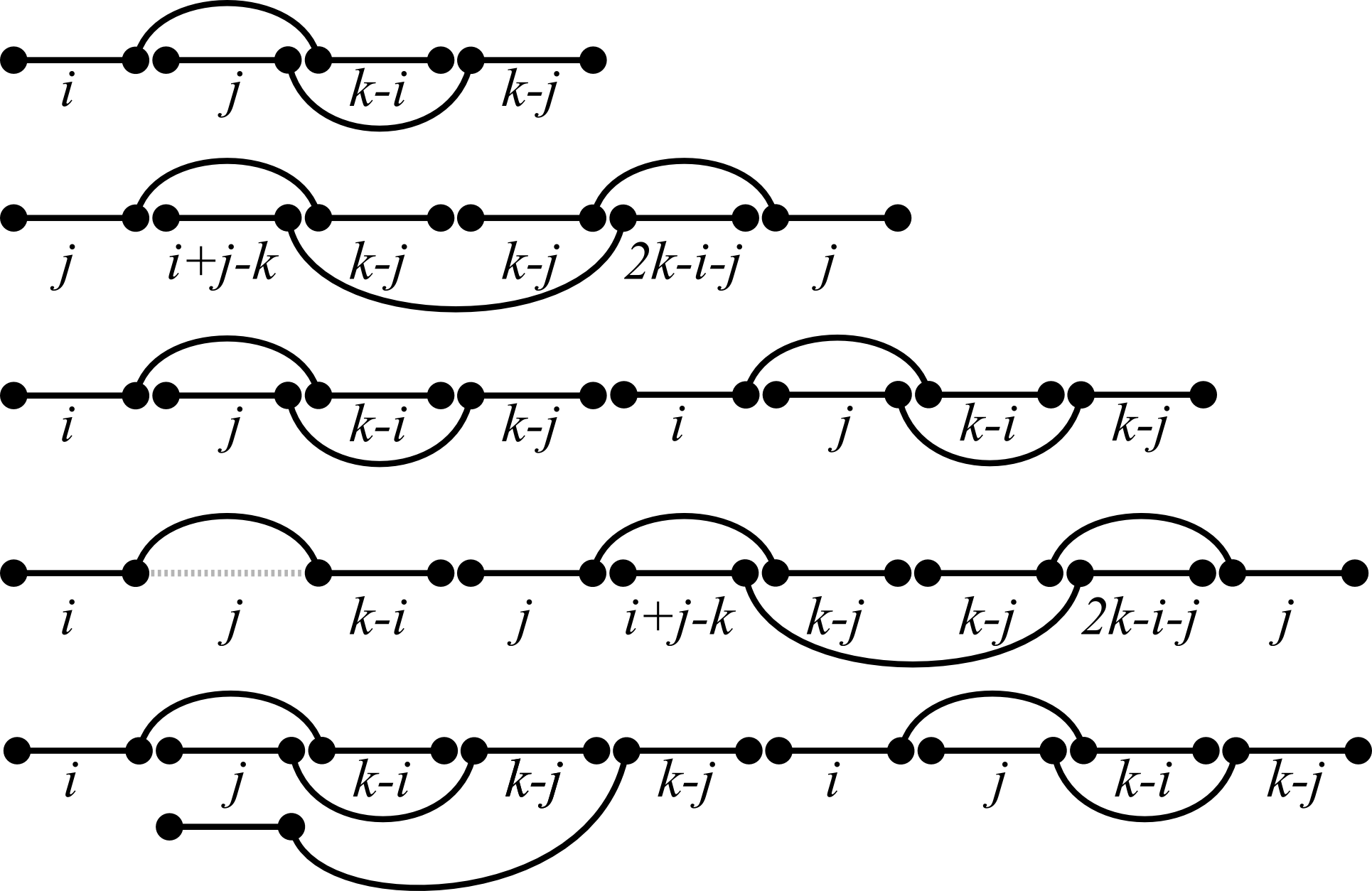}
  \caption{The edges in $E(v,i,j)$  when we have $i+j \geq k$ and $j< i$.} \label{fig:evig2}
\end{figure}               
  
\newpage

If  $i+j < k$ and $i$ is even, then we let
\begin{align*}
E(v,i,j)=\{&e(v,i,j), e(v+k+j, k-\frac{i}{2},i+j),\\ 
&e(v+2k+j-\frac{i}{2}, i+j, i),  e(v,i+j,\frac{i}{2}), \\
&e(v+2k+i+2j, \frac{i}{2},k-\frac{i}{2}),
                e(v+i, j+\frac{i}{2}, k-i-j), \\
&e(v+2k-j, k-\frac{i}{2},i+j), e(v+3k-j-\frac{i}{2}, i+j, i),\\ 
&e(v+3k+i, \frac{i}{2},k-\frac{i}{2}),          e(v, k-\frac{i}{2},i+j),\\
&e(v + k-\frac{i}{2}, i+j, i), e(v+k+i+j, \frac{i}{2},k-i-j),\\
                 &e(v+i, j, k-i), e(v+2k,i,j), e(v+2k+i,j,k-i)\}.
\end{align*} 
  \begin{figure}
  \centering
    \includegraphics[width=0.87\textwidth]{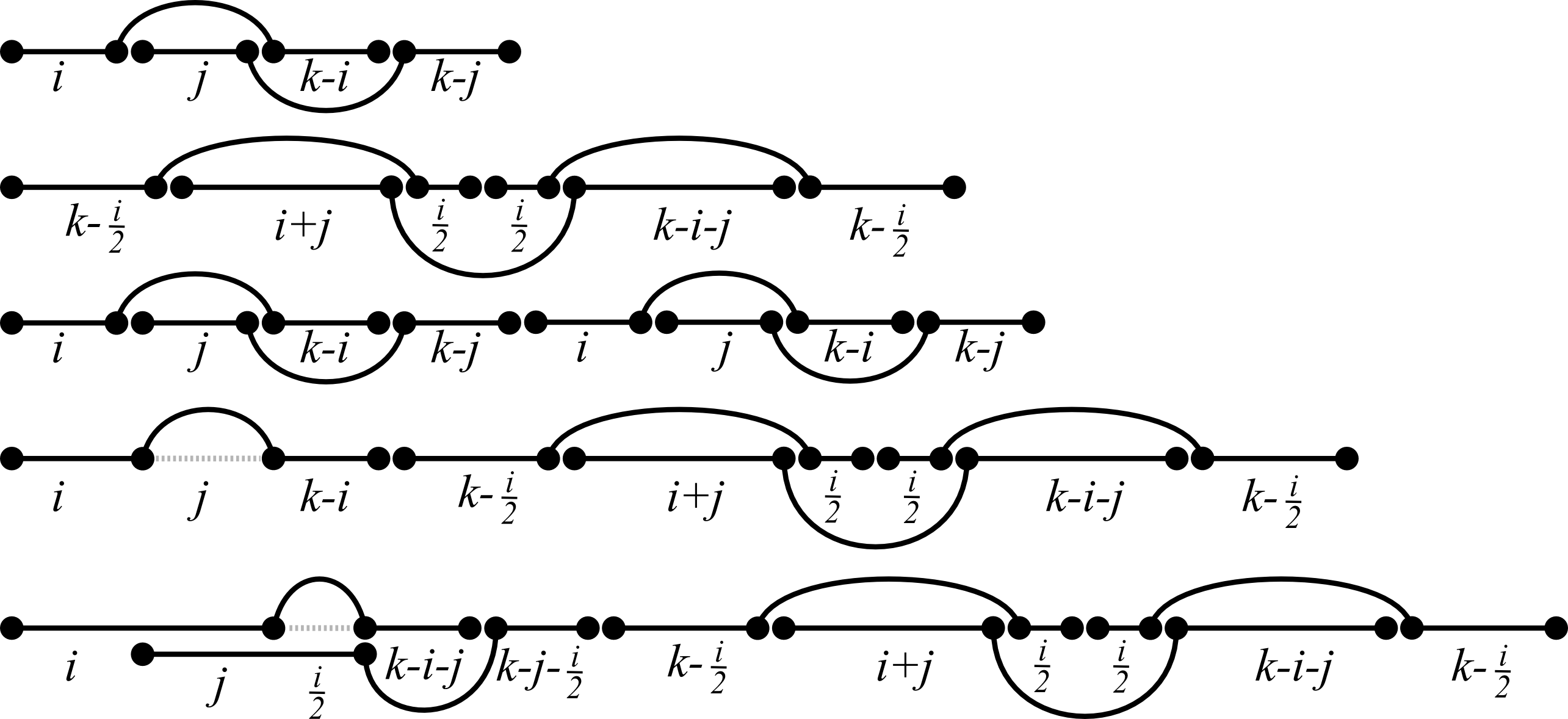}
  \caption{The edges in $E(v,i,j)$  when we have $i+j < k$ and $i$ is even.} \label{fig:evig3}
\end{figure}

If $i+j < k$ and $i$ is odd, then we let
\begin{align*}
E(v,i,j)=\{&e(v,i,j), e(v+k+j, k-\frac{i-1}{2},i+j), \\
&e(v+2k+j-\frac{i-1}{2}, i+j, i), e(v,i+j,\frac{i-1}{2}),\\
&e(v+2k+i+2j, \frac{i-1}{2},k-\frac{i-1}{2}),                 e(v+i, j+\frac{i-1}{2}, k-i-j),\\ 
&e(v+2k-j, k-\frac{i-1}{2},i+j), e(v+3k-j-\frac{i-1}{2}, i+j, i), \\
&e(v+3k+i, \frac{i-1}{2},k-\frac{i-1}{2}),                e(v, k-\frac{i-1}{2},i+j), \\
&e(v + k-\frac{i-1}{2}, i+j, i), e(v+k+i+j, \frac{i-1}{2},k-i-j),\\
                 &e(v+i, j, k-i), e(v+2k,i,j), e(v+2k+i,j,k-i)\}.
\end{align*} 

  \begin{figure}
  \centering
    \includegraphics[width=0.87\textwidth]{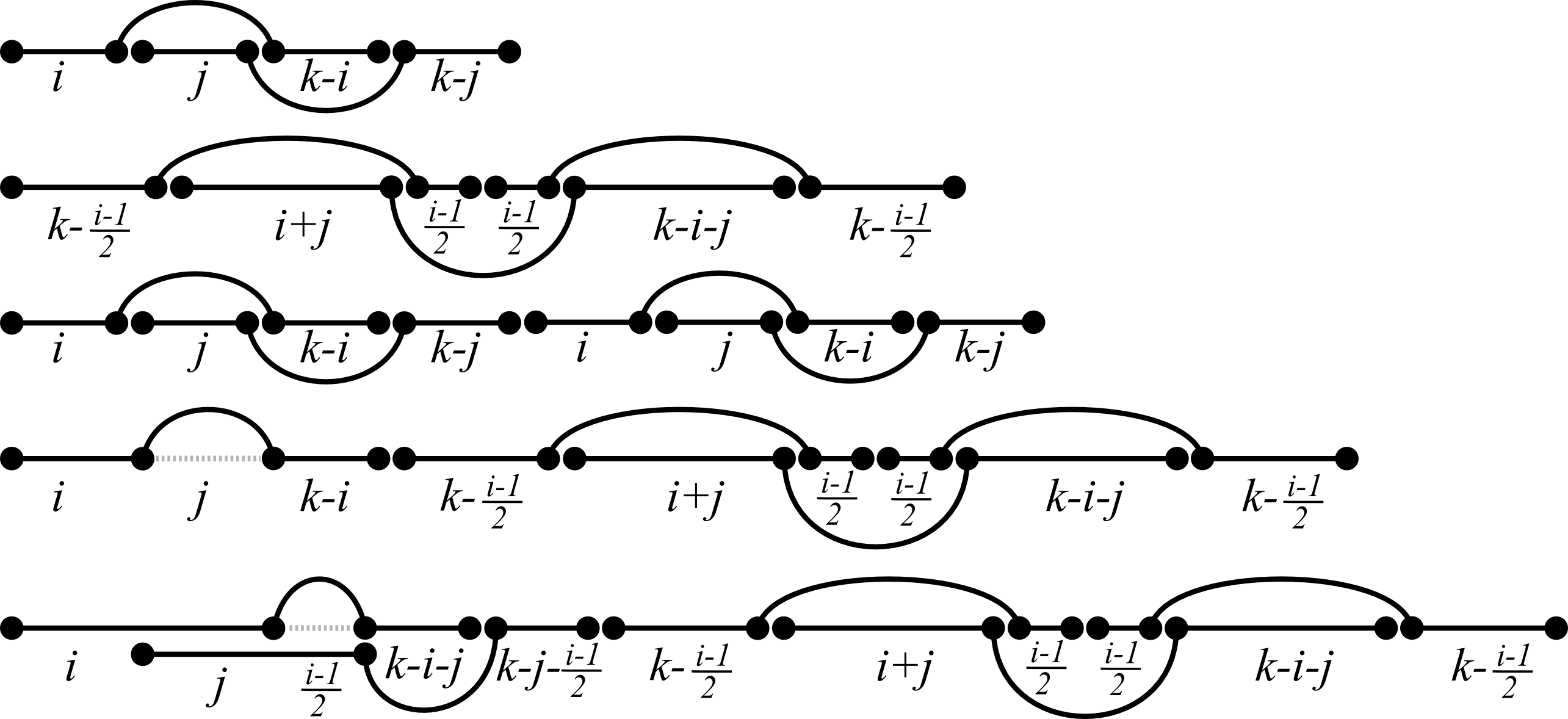}
  \caption{The edges in $E(v,i,j)$  when we have $i+j < k$ and $i$ is odd.} \label{fig:evig4}
\end{figure}

We define $E^-(v,i,j)$ to be the set of edges corresponding to edges in $E(v,i,j)$, but going anticlockwise (i.e. $E^-(v,i,j)=\{\{x_1, \dots, x_k\}:  \{v-(x_1-v), \dots, v-(x_k-v)\} \in E(v,i,j)).$
For each vertex $v$, we let $$E(v)=\bigcup_{i,j \in [1, k-1]}E(v,i,j)\cup E^-(v,i,j).$$

Notice that from the definition of $E(v,i,j)$, we certainly have $E(v,i,j)\leq 15$ for every $i, j\in [1,k-1]$, which implies that $|E(v)| \leq 15(k-1)^2$.  Also, since $e(v,i,j)\in E(v)$ for every $i,  j\in [1,k-1]$, we have that $E(v)\geq (k-1)^2$.  Therefore, we have $|E(v)|=\Theta(k^2)$.

There are only two features of the sets $E(v,i,j)$ that will be needed in the proof of Theorem~\ref{MMSMainTheorem}.  One is that sequences of edges similar to the ones in Figures~\ref{fig:evig1} -- \ref{fig:evig4} exist in $E(v,i,j)$.   This allows us to prove the following lemma.

\begin{lemma}\label{nonnegativeintervals}
Suppose that $i,j \in [1,k-1]$ and all the edges in $E(v,i,j)$ are negative.  The following hold.
\begin{enumerate}[\normalfont(i)]
\item $f([v, v+2k-1])<0$.
\item $f([v, v+3k-1])<0$.
\item $f([v, v+4k-1])<0$.
\item $f([v+i, v+i+j-1])<0 \implies f([v, v+4k+j-1])<0$.
\item $f([v+i, v+i+j-1])\geq 0 \implies f([v, v+5k-j-1])<0$.
\end{enumerate}
\end{lemma}
\begin{proof}
Figures~\ref{fig:evig1} -- \ref{fig:evig4} illustrates the constructions that are used in the proof of this lemma.
\begin{enumerate}[\normalfont(i)]
\item This follows from the fact that $e(v,i,j), e(v+i, j, k-i) \in E(v,i,j)$ and $e(v,i,j) \cup e(v+i, j, k-i) = [v, v+2k-1]$.

\item For $i+j\geq k$ and $i\geq j$, this  follows from the fact that $e(v, i, i+j-k), e(v+i, i+j-k, 2k-2i), e(v+i+j, k-i, 2k-i-j) \in E(v,i,j)$ and $e(v, i, i+j-k)\cup e(v+i, i+j-k, 2k-2i)\cup e(v+i+j, k-i, 2k-i-j) = [v, v+3k-1]$.  The other cases are similar.

\item This follows from the fact that $e(v,i,j), e(v+i, j, k-i), e(v+2k,i,j), e(v+2k+i, j, k-i) \in E(v,i,j)$ and $e(v,i,j) \cup e(v+i, j, k-i) \cup  e(v+2k,i,j)\cup e(v+2k+i, j, k-i)= [v, v+4k-1]$.

\item For $i+j\geq k$ and $i\geq j$, this  follows from the fact that $e(v,i,j), e(v+k+j,i,i+j-k), e(v+k+i+j, i+j-k, 2k-2i), e(v+k+i+2j, k-i,2k-i-j) \in E(v,i,j)$ and $e(v,i,j)\cup e(v+k+j,i,i+j-k)\cup e(v+k+i+j, i+j-k, 2k-2i)\cup e(v+k+i+2j, k-i,2k-i-j) \cup [v+i, v+i+j-1]= [v, v+4k+j-1]$.  The other cases are similar.

\item For $i+j\geq k$ and $i\geq j$, this  follows from the fact that $ e(v,i,j), e(v+i, j, k-i), e(v+i, j, 2k-i-j), e(v+3k-j, i, j), e(v+3k - j+i, j, k-i) \in E(v,i,j)$ and also $ e(v,i,j)\cup e(v+i, j, k-i)\cup e(v+i, j, 2k-i-j)\cup e(v+3k-j, i, j)\cup e(v+3k - j+i, j, k-i)  = [v, v+5k-j-1]$ and $e(v+i, j, k-i)\cap e(v+i, j, 2k-i-j)= [v+i, v+i+j-1]$.  The other cases are similar.
\end{enumerate}
\end{proof}

 The other feature of the sets $E(v,i,j)$ that we need is that no edge is contained in too many of the sets $E(v,i,j)$.  This is quantified in the following lemma.
For the duration of this \body, we fix the constant $C_1=110$.
\begin{lemma} \label{spreadout}
Let $e$ be an edge in $\Hyper{H}_{n,k}$.  The edge $e$ is contained in at most $C_1$ of the sets $E(v,i,j)\cup E^-(v,i,j)$ for $v \in V(\Hyper{H}_{n,k})$, and $i,j\in [1,k-1]$.
\end{lemma}
\begin{proof}
Notice that there are 55 edges mentioned in the definition of $E(v,i,j)$.
For $t=1, \dots, 55$, let $F^t(v,i,j)$ be the singleton containing the $t$th edge in the definition of $E(v,i,j)$, i.e. $F^1(v,i,j)=\{e(v,i,j)\}$, $F^2(v,i,j)= \{e(v+k+j,i,i+j-k)\}, \dots, F^{55}(v,i,j)=\{e(v+2k+i, j, k-i)\}$.
This definition is purely formal---for certain $i$ and $j$, it is possible that an edge in $F^t(v,i,j)$ is not an edge of $\Hyper{H}_{n,k}$ (for example $F^3(v,i,j)$ contains the edge $e(v+k+i+j, i+j-k, 2k-2i)$ which is not an edge of $\Hyper{H}_{n,k}$ if $2k-2i\geq k$).  Similarly it is possible for $F^t(v,i,j)$ to be empty for certain $i$ and~$j$---for example $F^{52}(v,i,j)$ should contain $e(v+k+i+j, \frac{i-1}{2}, k-i-j)$ which is not defined when $i$ is even.

Clearly $E(v,i,j)\subseteq \bigcup_{t=1} ^{55} F^t(v,i,j)$  holds. 
 Also, it is straightforward to check that for fixed $t$, the sets $F^t(v,i,j)$ are all disjoint for $v \in V(\Hyper{H}_{n,k})$, and $i,j\in [1, k-1]$.  
Indeed for fixed $t$, if we have $e(u,a,b)\in F^t(v,i,j)$, then it is always possible to work out $v, i$, and $j$ uniquely in terms of $u$, $a$, and $b$.
 These two facts, together with the Pigeonhole Principle imply that the edge $e$ can be contained in at most  $55$ of the sets $E(v,i,j)$ for $v \in V(\Hyper{H}_{n,k})$, and $i,j\in [1, k]$.  The lemma follows, since $C_1\geq 2\cdot 55 =110$.
\end{proof}

A useful corollary of Lemma~\ref{spreadout} is that an edge $e$ can be contained in at most $110$ of the sets $E(v)$ for  $v\in V(\Hyper{H}_{n,k})$.

%% file: MMSProofSimple.tex
In this section we prove Theorem~\ref{MMSMainTheorem}, with a weaker bound of $n\geq 14k^4$.  This proof has many of the same ideas as the proof of Theorem~\ref{MMSMainTheorem}, but is much shorter.  We therefore present it in order to illustrate the techniques that we will use in proving Theorem~\ref{MMSMainTheorem}, and hopefully aid the reader to understand that theorem.
\begin{theorem}\label{MMSQuartic}
For $n\geq 14k^4$, and every function $f:V(\Hyper{H}_{n,k})\to \mathbb{R}$ which satisfies $\sum_{x \in V(\Hyper{H}_{n,k})} f(x)\geq 0$ we have $e_f ^+(\Hyper{H}_{n,k})\geq k(k-1)^2$.
\end{theorem}
\begin{proof}
Suppose for the sake of contradiction that we have a function $f:V(\Hyper{H}_{n,k})\to \mathbb{R}$ satisfying  $\sum_{x \in V(\Hyper{H}_{n,k})} f(x)\geq 0$ such that we have $e^+_f(\Hyper{H}_{n,k})< k(k-1)^2$.

The proof of the theorem rests on two claims.  The first of these says that any sufficiently small interval $I$ in $\mathbb{Z}_n$ is contained in a negative interval of almost the same order as $I$.
\begin{claim} \label{coverSimple}
Let $I$ be an interval in $\mathbb{Z}_n$ such that $|I|\leq n-2k$.  Then there is an interval $J= [j_1, j_t]$ which satisfies the following:
\begin{enumerate}[\normalfont(i)]
\item $|J|\leq |I|+2k$.
\item $I\subseteq J$.
\item $ f(J) < 0$.
\end{enumerate}
\end{claim}
\begin{proof}
Without loss of generality, we may assume that $I$ is the interval $[2k, 2km+l]$ for some $l\in[0,2k-1]$ and $m\leq\frac{n}{2k}-1$.  First we will exhibit $2k(k-1)^2$ sets of vertex-disjoint edges covering $I$.

For  $v \in \{0\dots 2k-1\}$, $i, j\in\{1, \dots, k-1\}$ we let
\begin{align*}
\Hyper{D}(v,i,j) &= \bigcup_{t=0}^m \Big( e(v+2tk,i,j)\cup e(v+2tk+i,j, k-i) \Big)
\end{align*}
Notice that an edge $e(u,a,b)$ is contained only in the sets $\Hyper{D}(u\pmod{2k},a,b)$ and $\Hyper{D}(u-k+b\pmod{2k},k-b,a)$.  Therefore, since there are at less than $k(k-1)^2$ nonnegative edges in $\Hyper{H}_{n,k}$, there are some $v_0$,$i_0$ and $j_0$ for which the set $\Hyper{D}(v_0,i_0,j_0)$ contains only negative edges.  Letting $J=\bigcup  \Hyper{D}(v_0,i_0,j_0) =[v_0,v_0+2k(m+1)]$ implies the claim.
\end{proof}

The second claim that we need shows that any sufficiently large interval which does not contain nonnegative edges in $\Hyper{H}_{n,d}$ must be negative.
\begin{claim} \label{transverseSimple}
Let $I= [i_1,  i_m]$ be an interval in  $\mathbb{Z}_n$ which satisfies the following:
\begin{enumerate}[\normalfont(i)]
\item $|I|\geq 12k$.
\item There are no nonnegative edges of $\Hyper{H}_{n,k}$ contained in $I$.
\end{enumerate}
We have that $f(I) < 0$.
\end{claim}
\newcommand{\Qpos}{Q^-}
\newcommand{\Qneg}{Q^+}

\begin{proof}
Let $R_0=\{v\in I: f([0,v-1])<0\}$ and $R_m=\{v\in I: f([v,m])<0\}$.
Let $\Qpos=\{i\in[1,k-1]: f([1,i])< 0\}$ and $\Qneg=\{k-i\in[1,k-1]: f([1,i])\geq 0\}$.

Since $I$ contains only negative edges, parts (iv) and (v) of Lemma~\ref{nonnegativeintervals} imply that we have that $(\Qpos\cup \Qneg)+4k\subseteq R_0$.  Part (iii) of Lemma~\ref{nonnegativeintervals} implies that $4k\in R_0$.
Then, parts (i) and (ii) of Lemma~\ref{nonnegativeintervals} imply that $(\Qpos\cup \Qneg\cup \{0\})+tk\subseteq R_0$ for any $t \in\left\{6,7,\dots, \left\lfloor\frac{m}{k}\right\rfloor-1\right\}$.  This implies that we have $R_0\cap[u,u+k-1]\geq |\Qpos\cup \Qneg\cup \{0\}|$ for any $u\in[6k,m-k-1]$.

Notice that $\Qpos\cup \Qneg$ contains at least one element from each of the sets $\{1, k-1\}, \dots, \left\{\left\lfloor\frac{k}{2}\right\rfloor, \left\lceil\frac{k}{2}\right\rceil\right\}$.  This implies that for every $u\in\{6k,\dots,m-k-1\}$ we have 
$$|R_0\cap [u, u+k-1]|\geq |\Qpos\cup \Qneg\cup \{0\}|\geq \left\lfloor\frac{k}{2}\right\rfloor+1>\frac{k}{2}.$$
Similarly we obtain $|R_m\cap [u, u+k-1]|>\frac{k}{2}$ for every $u \in \{k,\dots, m-7k\}$.  By choosing $u=6k$, we have that $|R_0\cap [6k, 7k-1]|,|R_m\cap [6k, 7k-1]|>\frac{k}{2}$, and hence there exists some $i\in[6k,7k-1]$ such that $i\in R_0,  R_m$ hold.  This gives us $f([0,m])=f([0,i])+f([i+1,m])<0$, proving the claim.
\end{proof}
We now prove the theorem.   
Suppose that every interval of order $14k$ in $\Hyper{H}_{n,k}$ contains a nonnegative edge.  Since there are at least $\frac{n}{14k}\geq k^3$ such disjoint intervals in $\Hyper{H}_{n,k}$, we have at least $k^3$ nonnegative edges in $\Hyper{H}_{n,k}$, contradicting our initial assumption that $e^+_f(\Hyper{H}_{n,k})< k(k-1)^2$.

Suppose that there is an interval $I$ of order $14k$ in $\Hyper{H}_{n,k}$ which contains only negative edges.  Applying Claim~\ref{coverSimple} to $V(\Hyper{H}_{n,k}\setminus I)$ we obtain an interval $J\subseteq I$ such that $f(V({\Hyper{H}_{n,k})}\setminus J)<0$ and $|J|\geq 12k$.  Applying Claim~\ref{transverseSimple} to $J$ we obtain that $f(J)<~0$. Therefore, we have $f(V(\Hyper{H}_{n,k}))= f(J)+f(V(\Hyper{H}_{n,k})\setminus J)<0$ contradicting the assumption that $f(V(\Hyper{H}_{n,k}))\geq 0$ in the theorem
\end{proof}

It is not hard to see that Claim~\ref{transverseSimple} would still be true if we allowed $I$ to contain a small number of nonnegative edges.  The proof of Theorem~\ref{MMSMainTheorem} is similar to the proof of Theorem~\ref{MMSQuartic} since it also consists of two main claims which are analogues of Claims~\ref{coverSimple} and~\ref{transverseSimple}.  However the analogue of Claim~\ref{transverseSimple} is much stronger since it allows for $O(k^3)$ nonnegative edges to be contained in $I$.  This is the main improvement in the proof of Theorem~\ref{MMSQuartic} which is needed to obtain the linear bound which we have in Theorem~\ref{MMSMainTheorem}.

%% file: MMSProofNew.tex
\newcommand{\ethree}{\epsilon_3}
\newcommand{\efour}{(7\epsilon_1+ \epsilon_2)}
\newcommand{\eQSum}{\ethree^2}
\newcommand{\esix}{\epsilon_4}
\newcommand{\eFirstPow}{\epsilon_5}

In this section we use ideas from Sections~\ref{MMSConstructionSection} and~\ref{MMSSimpleProofSection} in order to Theorem~\ref{MMSMainTheorem}.
\begin{proof}[Proof of Theorem~\ref{MMSMainTheorem}]

For convenience, we fix the following constants for the duration of the proof.
$$
\begin{array}{lcl}
C_0=10^{46} & &
\epsilon_0=  10^{-9} \\
C_1=110& &
\epsilon_1=  10^{-18} \\
C_2=10^{16}& &
\epsilon_2= 10^{-6} \\
C_3=28 & &
\epsilon_3= 10^{-2} \\
 & &
\epsilon_4= 0.1 \\
& & 
\epsilon_5=0.25
\end{array}
$$

Let $n\geq C_0 k$, and let $\Hyper{H}_{n,k}$ be the hypergraph defined in Section~\ref{MMSConstructionSection}.
Recall that for any vertex $v\in V(\Hyper{H}_{n,k})$, we have $|E(v)|= \Theta(k^2)$.

\begin{definition}
We say that a vertex $v$ in $\Hyper{H}_{n,d}$ is \textbf{bad} if at least $\epsilon_0 k^2$ of the edges in $E(v)$ are nonnegative and \textbf{good} otherwise. 
\end{definition}

Let $G_{\Hyper{H}}$ be the set of good vertices in $\Hyper{H}_{n,k}$.

Suppose that we have a function $f:V(\Hyper{H}_{n,k})\to \mathbb{R}$ such that we have $e^+ _f(\Hyper{H}_{n,k})<k(k-1)^2$.  We will show that $f(V(\Hyper{H}_{n,k}))<0$ holds.
The proof of the theorem consists of the following two claims.

\begin{claim} \label{cover}
Let $I$ be an interval in $\mathbb{Z}_n$ such that $|I|\leq n-4C_2k$.  There is an interval $J= [j_1, j_t]$ which satisfies the following:
\begin{enumerate}[\normalfont(i)]
\item $|J|\leq |I|+4C_2k$.
\item $I\subseteq J$.
\item Both $j_1-1$ and $j_t+1$ are good.
\item $ f(J) < 0$.
\end{enumerate}
\end{claim}

\begin{claim} \label{transverse}
Let $I= [i_1,  i_m]$ be an interval in  $\mathbb{Z}_n$ which satisfies the following:
\begin{enumerate}[\normalfont(i)]
\item $C_3 k \leq |I| \leq (C_3 + 4C_2) k$.
\item Both $i_1$ and $i_m$ are good.
\item Every subinterval of $I$ of order $k$,  contains at most $\epsilon_1 k$ bad vertices.
\end{enumerate}
We have that $f(I) < 0$.
\end{claim}

Once we have these two claims, the theorem follows easily:

First suppose that no intervals in $\mathbb{Z}_n$ of order $(C_3+4C_2)k$ satisfies condition (iii) of Claim~\ref{transverse}.  This implies that there are at least $\epsilon_1 C_0 k/(C_3+4C_2)$ bad vertices in $\Hyper{H}_{n,k}$.  Then Claim~\ref{spreadout} together with the definition of ``bad'' implies that there are at least $\epsilon_0 \epsilon_1 C_0k^3/C_1(C_3+4C_2)$ nonnegative edges in $\Hyper{H}_{n,k}$.  However, since $\epsilon_0 \epsilon_1 C_0/C_1(C_3+4C_2)\geq 1$, this contradicts our assumption that $e^+ _f(\Hyper{H}_{n,k})<k(k-1)^2$.

Now, suppose that there is an interval $I$ of order $(C_3+4C_2) k$ which satisfies condition~(iii) of Claim~\ref{transverse}.  
Notice that all subintervals of $I$ will also satisfy condition (iii) of Claim~\ref{transverse}.  
Applying Claim~\ref{cover} to $V(\Hyper{H}_{n,k}) \setminus I$ gives an interval $J\subseteq I$ which satisfies all the conditions of Claim~\ref{transverse} and also $ f(V(\Hyper{H}_{n,k})\setminus J)<0$.  
Applying Claim~\ref{transverse} to $J$ implies that we also have $f(J)<0$. 
We have $\sum_{v\in \Hyper{H}_{n,k}} f(v)= f(V(\Hyper{H}_{n,k})\setminus J)+f(J)<0$, contradicting our initial assumption and proving the theorem.

It remains to prove Claims \ref{cover} and \ref{transverse}.

\begin{proof}[Proof of Claim~\ref{cover}]

Without loss of generality, we may assume that $I$ is the interval $[0, 2km+l]$ for some $l\in[0,2k-1]$ and $m<\frac{n}{2k}-2C_2$. 
We partition  $[1, 2k]$ into two sets as follows.

\begin{definition}
For $r \in [1, 2k]$ we say that $r$ is \textbf{unblocked} if for every $t\in [-C_2,  m+C_2]$, there are some $i,j \in [1,k-1]$ such that both of the edges   $e(2tk+r,i,j)$ and $e(2tk+r+i,j,k-i)$  are negative.
We say that $r$ is \textbf{blocked} otherwise.
\end{definition}

Notice that if $r$ is unblocked, then for every $t_1\in [-C_2, 0]$ and $ t_2 \in [m, m+C_2 ]$ we have that $f([2t_1k+r, 2t_2k+r-1])<0$.  Therefore the claim holds unless either $2t_1k+r-1$ or $2t_2k+r$ is bad.  Therefore, for each $r$ which is unblocked, we can assume that all the vertices in either $\{r-1-2kC_2,r-1-2k(C_2-1), \dots, r-1\}$ or $\{r+2km, r+2k(m+1), \dots, r+2k(m+C_2)\}$ are bad.

To each $r\in  [1,  2k]$, we assign a set of nonnegative edges, $P(r)$, as follows:
\begin{itemize}
\item If $r$ is blocked, then there is some $t_r\in [-C_2, m+C_2 ]$, such that for every $i,j\in [1,k-1]$ one of the edges $e(2t_rk+r,i,j)$ or $e(2t_rk+r+i,j,k-i)$ is nonnegative.  We let $P(r)$ be the set of these edges.  Notice that this ensures that $|P(r)|\geq(k-1)^2$.  Also, note that for fixed $a$,$b$,$c$ the $P(r)$ can contain at most one edge of the form $e(a+2tk, b,c)$ for any $t\in [-C_2, m+C_2 ]$.

\item If $r$ is unblocked we know that  all the vertices in either $\{r-1-2kC_2,r-1-2k(C_2-1), \dots, r-1\}$ or $\{r+2km, r+2k(m+1), \dots, r+2k(m+C_2)\}$ are bad.  Let $P(r)$ be the set of nonnegative edges in $E(r-1-2kC_2)\cup E(r-1-2k(C_2-1))\cup \dots \cup E(r-1) \cup E(r+2km) \cup E(r+2k(m+1))\cup \dots \cup E(r+2k(m+C_2))$.  Since at least $C_2$ of these vertices are bad, Lemma~\ref{spreadout} together with the Pigeonhole Principle implies that $|P(r)|\geq \frac{C_2 \epsilon_0}{C_1}  k^2$.
\end{itemize}

Notice that an edge $e$ can be in at most $2$ of the sets $P(r)$ for $r$ blocked.  This is because it can be in at most one such set as an edge of the form ``$e(tk+r,i,j)$" and in at most one such set and as an edge of the form ``$e(tk+r+i,j,k-i)$''.
Therefore we have:
\begin{equation} \label{eqblocked}
  \left|\bigcup_{r \text{ blocked}} P(r)\right| \geq  \sum_{r \text{ blocked}}\frac{1}{2}(k-1)^2
\end{equation}

Lemma~\ref{spreadout} implies that an edge $e$ can be in at most $C_1$ of the sets $P(r)$ for $r$ unblocked.  Therefore we have:
\begin{equation}\label{equnblocked}
  \left|\bigcup_{r \text{ unblocked}} P(r)\right| \geq   \sum_{r \text{ unblocked}} \frac{C_2 \epsilon_0}{(C_1)^2}k^2 
\end{equation}

We claim that for any $s\in[1, 2k]$, we have 
\begin{equation}\label{eqEcapP}
\left|\left(\bigcup_{t_\in[-C_2,m+C_2]} E(s+2tk) \right) \cap \left(\bigcup_{r \text{ blocked}} P(r) \right)\right|\leq  2|E(s)| .
\end{equation}

  Indeed, otherwise the Pigeonhole Principle implies that for some $r\in[1, 2k]$, $t_1,$ $t_2,$ $t_3 \in [-C_2,m+C_2]$, and $i, j \in [1,k-1]$ we have three distinct edges $e(r+2t_1k, i, j)$,  $e(r+2t_2k,i,j)$, and $e(r+2t_3k,i,j)$ which are are all contained in $\left(\bigcup_{t_\in[-C_2,m+C_2]} E(s+2tk) \right) \cap \Big(\bigcup_{r \text{ blocked}} P(r) \Big)$.  
This means that there are some $r_1$, $r_2$, and $r_3 \in [1,2k]$ which are blocked, such that $e(r+2t_lk, i, j)\in P(r_l)$ holds for $l = 1$, $2$ and $3$.  
Since each  $r_l$ is blocked, all the edges in $P(r_l)$ are of the form $e(2t'k+r_l,i',j')$ or $e(2t'k+r_l+i',j',k-i')$ for some $t' \in [-C_2, m+C_2]$ and $i', j' \in [1,k-1]$.  This, together with  $e(r+2t_lk, i, j)\in P(r_l)$, implies that we have $r_1, r_2, r_3 \in \{r, r-k+j\}$.  
This means that for some distinct $l, l' \in \{1,2,3\}$, we have $r_l=r_{l'}$, which means that both $e(r+2t_lk, i, j)$ and $e(r+2t_{l'}k, i, j)$ are contained in $P(r_l)$.  However, this contradicts our definition of $P(r_l)$ for $r_l$ blocked which allowed only one edge of the form $e(r+2tk, i, j)$ to be in $P(r_l)$ for fixed $r$, $i$ and~$j$.  This shows that (\ref{eqEcapP}) holds for all $s\in[1,2k]$.

Recall that for all vertices $s$ we have $|E(s)|\leq C_1k^2$.  This, together with (\ref{eqEcapP}) implies that we have 
\begin{align}
\left|\left(\bigcup_{s \text{ unblocked}} P(s) \right) \cap \left(\bigcup_{r \text{ blocked}} P(r) \right)\right| &\leq 
 \left|\left(\bigcup_{\substack{s \text{ unblocked}, \\ t_\in[-C_2,m+C_2]}} E(s+2tk) \right) \cap \left(\bigcup_{r \text{ blocked}} P(r) \right)\right| \notag \\
&\leq \sum_{s \text{ unblocked}} 2|E(s)| \notag \\
&\leq \sum_{s \text{ unblocked}} 2C_1 k^2.  \label{eqintersection}
\end{align}

Putting (\ref{eqblocked}), (\ref{equnblocked}), and (\ref{eqintersection}) together, we obtain:
\begin{align} 
e^+ _f(\Hyper{H}_{n,k}) &\geq  \left|\bigcup_{r \text{ blocked}} P(r)\right|+ \left|\bigcup_{r \text{ unblocked}} P(r)\right| - \left|\left(\bigcup_{s \text{ unblocked}} P(s) \right) \cap \left(\bigcup_{r \text{ blocked}} P(r) \right)\right| \notag\\
&\geq \sum_{r \text{ blocked}} \frac{1}{2}(k-1)^2 +  \sum_{r \text{ unblocked}} \frac{C_2 \epsilon_0}{(C_1)^2}k^2  - \sum_{s \text{ unblocked}} 2C_1 k^2 \notag\\
&\geq \sum_{r \text{ blocked}} \frac{1}{2}(k-1)^2 +  \sum_{r \text{ unblocked}} \frac{1}{2} k^2 \notag\\
&\geq k(k-1)^2. \label{toomanyedges}
\end{align}

The second last inequality follows from $\frac{C_2 \epsilon_0}{(C_1)^2}-2C_1\geq \frac{1}{2}$.  The last inequality follows from the fact that ``the number of blocked vertices" $ + $ ``the number of unblocked vertices" $ = 2k$.  However (\ref{toomanyedges}) contradicts the assumption that there are less than $ k(k-1)^2$ nonnegative edges in $\Hyper{H}_{n,k}$, proving the claim.
\end{proof}

It remains to prove Claim~\ref{transverse}.

\begin{proof}[Proof of Claim~\ref{transverse}]

Without loss of generality, we can assume that $I= [0, m ]$ for some $m\leq (C_3+4C_2)k$.

Recall that we are using notation from additive combinatorics for sumsets and translates.  Except where otherwise stated, sumsets will lie in $\mathbb{Z}$. 
For a set $A\subseteq \mathbb{Z}$, define $$A\bmod(k)=\{b\in[0,k-1]: b\equiv a \bmod(k) \text{ for some } a \in A\}.$$

For each vertex $v$, we define a set of vertices $R(v)$ contained in $I$.
$$R(v)= \{u \in [v+1,m] : f([v, u-1]) < 0 \text{ and } u \text{ is good.}\} $$

$R(v)$ has the following basic properties.
\begin{claim} \label{basicR}
The following hold. 
\begin{enumerate}[\normalfont(i)]
\item If $u>v$ and $u\in R(v)$, we have  $R(u)\subseteq R(v)$.
\item Suppose  that $t\geq 2$ and we have a set $X\subseteq R(v)\cap [w, w+2k-1]$, for some vertex $w$.  There is a subset $X'\subseteq X$, such that we have $|X'|\geq |X|- 2\epsilon_1 kt$ and $X'+t'k \subseteq R(v)$ for every $t'\in\{2, \dots, t\}$.
\item Suppose that we have $X\subseteq [0, 2k-1]$ such that $X+t_0k \subseteq R(0)$ for some $t_0$.  There is a subset $X'\subseteq X \bmod(k)$, such that $X'+ (t_0+3)k \subseteq R(0)$ and $|X'|\geq |X|- 6\epsilon_1 k$. 
\item Suppose that we have $X\subseteq [w, w+k-1]\cap R(0)$ for some $w$.  Then for any $v\geq w+2k$, we have  we have  $|R(0) \cap [v,v+k-1]|\geq |X|-2\epsilon_1(v-w+1)k$.
\end{enumerate}
\end{claim}

\begin{proof}
\begin{enumerate}[\normalfont(i)]
\item This part is immediate from the definition of $R(v)$.

\item First, we deal with the case when $t=2$ or $3$.  The general case will follow by induction.

Suppose that we have $x\in X$.
Since $x$ is good, Lemma~\ref{spreadout} implies that there are at most $\epsilon_0 C_1 k^2$ pairs $i$, $j$ for which $E(x,i,j)$ contains a nonnegative edge.  Therefore, since $\epsilon_0 C_1<1$, there must be at least one pair $i_0,j_0$ for which all the edges in $E(x,i_0,j_0)$ are nonnegative.  Combining this with parts (i) and (ii) of Lemma~\ref{nonnegativeintervals} implies that we have 
\begin{equation}\label{2and3negative}
f([v,x+2k-1]), f([v,x+3k-1])<0.
\end{equation}

If $t=2$ we let $X'= X\cap(G_{\Hyper{H}}-2k)$.   The identity \ref{2and3negative} implies that $X'+2k\subseteq R(v)$.  By condition (iii) of Claim~\ref{transverse}, we know that there are at most $2\epsilon_1  k$ bad vertices in $[w+2k,  w+4k-1]$, which implies that $|X'|\geq |X|- 2\epsilon_1  k $.

Similarly, if $t=3$ we let $X'= X\cap(G_{\Hyper{H}}-2k)\cap(G_{\Hyper{H}}-3k)$.   The identity \ref{2and3negative} implies that $X'+2k, X'+3k\subseteq R(v)$.  By condition (iii) of Claim~\ref{transverse}, we know that there are at most $3\epsilon_1  k$ bad vertices in $[w+2k,  w+5k-1]$, which implies that $|X'|\geq |X|- 3\epsilon_1  k $.

Suppose that the claim holds for $t=t_0$ for some $t_0\geq 3$.  We will show that it holds for $t=t_0+1$.  We know that there is a set $X'\subseteq X + t_0k$, such that we have $|X'|\geq |X|- \epsilon_1  kt_0 $ and $X'+t'k \subseteq R(v)$ for  $t'=2, \dots, t_0$.  Applying the $t=2$ part of this claim to $X'+t_0k$ we obtain a set $X''\subseteq X'$ such that $|X''|\geq  |X'|- \epsilon_1   k \geq |X|-  \epsilon_1  k(t_0+1) $ and also $X''+(t_0+1)k\subseteq R(v)$.  This proves the claim by induction.

\item Apply part (i) to $X+t_0$ with $t=3$ to obtain a set $X'$ with $|X'|\geq |X|- 3\epsilon_1 k$ and $X'+t_0k+\{2k,3k\}\subseteq R(0)$.  Let $X''=X'\bmod(k)$ to obtain a set satisfying $X''\subseteq X\bmod(k)$ and $|X''|\geq |X\bmod(k)|- 3\epsilon_1 k$.  We have that $X''+t_0+3k= (X'\cap [0,k-1]+t_0+3k)\cup (X'\cap[k,2k-1]+t_0+2k)\subseteq X'+t_0+\{2k,3k\}\subseteq R(0)$.

\item Apply part (i) to $X$ with $t=\left\lfloor\frac{v-w}{k}\right\rfloor+1$ to obtain a set $X'$ with $|X'|\geq |X|- \epsilon_1 \left(\left\lfloor\frac{v-w}{k}\right\rfloor+1\right)k$ and $X'+t'k\subseteq R(0)$ for any $t'=2, \dots, \left(\left\lfloor\frac{v-w}{k}\right\rfloor +1\right) k$.  For any $x\in X'$, either $x+\left\lfloor\frac{v-w}{k}\right\rfloor k$ or $x+ \left(\left\lfloor\frac{v-w}{k}\right\rfloor +1\right)k$ is in $[v,v+k-1]\cap R_0$, which implies that $|R(0) \cap [v,v+k-1]|\geq |X'|\geq |X|-\epsilon_1(v-w+1)k$.
\end{enumerate}
\end{proof}

\newcommand{\Qp}[2]{Q^-_{#1}(#2)}
\newcommand{\Qm}[2]{Q^+_{#1}(#2)}
\newcommand{\Qe}[2]{Q_{#1}(#2)}

To every vertex $v\in I$ and $\epsilon>0$, we assign sets $\Qm{\epsilon}{v}$, $\Qp{\epsilon}{v}$, $\Qe{\epsilon}{v} \subseteq [1, k-1]$ as follows.

\begin{align*}
\Qp{\epsilon}{v}&= \{j \in [1,k-1]:  f([v+i,v+i+j-1])<0 \\
& \ \ \ \ \ \ \ \ \ \ \ \ \ \ \ \ \ \ \ \ \ \ \ \ \ \ \ \ \ \ \ \ \ \ \text{ for at least } \epsilon k \text{ numbers } i\in [1, k-1]\} \\
\Qm{\epsilon}{v}&= \{k-j \in [1,k-1]:  f([v+i,v+i+j-1])\geq 0 \\
& \ \ \ \ \ \ \ \ \ \ \ \ \ \ \ \ \ \ \ \ \ \ \ \ \ \ \ \ \ \ \ \ \ \ \text{ for at least } \epsilon k \text{ numbers } i\in [1, k-1]\}\\ 
\Qe{\epsilon}{v}&=\Qp{\epsilon}{v}\cup \Qm{\epsilon}{v}\cup\{0\}. 
\end{align*}

$\Qe{\epsilon}{v}$ has the following basic properties.
\begin{claim} \label{basicQ}
The following hold.
\begin{enumerate}[\normalfont(i)]
\item For any $r\in[0,k]$, we have $\Qe{2\epsilon}{v}\subseteq \Qe{\epsilon}{v-r}\cup \Qe{\epsilon}{v-r+k}$.
\item For $\epsilon\leq\frac{1}{2}$, $x\in  [1,k-1]$, and $v\in I$ either $x$ or $k-x$ is in $\Qe{\epsilon}{v}$.
\item For $\epsilon\leq\frac{1}{2}$ and $v\in I$, we have $|\Qe{\epsilon}{v}|\geq \frac{1}{2}k$.
\end{enumerate}
\end{claim}
\begin{proof}
If $j\in \Qp{2\epsilon}{v}$, then there are at least $2\epsilon k$ numbers $i\in [1, k-1]$ for which $f([v+i,v+i+j-1])<0$.  For every $r\in[0,k]$ the Pigeonhole Principle implies that there must either be at least $\epsilon k$ numbers $i\in [1, k-1]$ for which $f([v-r+i,v-r+i+j-1])<0$ or at least $\epsilon k$ numbers $i\in [1, k-1]$ for which $f([v-r+k+i,v-r+k+i+j-1])<0$.  Therefore we have $\Qp{2\epsilon}{v}\subseteq \Qp{\epsilon}{v-r}\cup \Qp{\epsilon}{v-r+k}$.  Similarly we obtain $\Qm{2\epsilon}{v}\subseteq \Qm{\epsilon}{v-r}\cup \Qm{\epsilon}{v-r+k}$ which implies part (i).

Part (ii) is immediate from the definition of $\Qe{\epsilon}{v}$.  Part (iii) follows from~(ii).
\end{proof}

The following claim shows that for a good vertex $v$, there is a certain translate of $\Qe{\eFirstPow}{v}$ which will nearly be contained in $R(v)$.

\begin{claim} \label{firstpower}
For any good vertex $v$ satisfying $0\leq v\leq  m-5k$, there is a $Q'\subseteq \Qe{\eFirstPow}{v}$ such that $|Q'|\geq  |\Qe{\eFirstPow}{v}| -\epsilon_2 k$ and we have 
\begin{equation*} \label{qe:fp0}
Q'+4k+ v\subseteq R(v).
\end{equation*}
\end{claim}
\begin{proof}
Let $T\subseteq [1,k-1]$ be the set of $j\in[1,k-1]$ for which there are at least $\eFirstPow k$ numbers $i \in[1,k-1]$ such that $E(v,i,j)$  contains a nonnegative edge.
We have at least $|T|\eFirstPow k$ pairs $i,j\in[1,k-1]$ for which $E(v,i,j)$ contains a nonnegative edge.
Since $v$ is good, Lemma~\ref{spreadout} implies that at most $\epsilon_0 C_1 k^2$ of the sets $E(v,i,j)$ contain nonnegative edges for $i,j\in [1,k-1]$.
Therefore, we have $|T|\eFirstPow k\leq \epsilon_0 C_1 k^2$.
We define the set $Q'$ as 
$$Q'=\left((\Qp{\eFirstPow}{v}\setminus T)\cup (\Qm{\eFirstPow}{v}\setminus T) \cup\{0\}\right)\cap (G_{\Hyper{H}}-4k).$$

First we prove $Q'+4k+v\subseteq R(v)$.
Suppose that we have $j\in \Qp{\eFirstPow}{v}\setminus T$. From the definition of $T$, there are at at more than $k-1-\eFirstPow k$ numbers $i \in[1,k-1]$ such that all the edges in $E(v,i,j)$  are negative.  From the definition of $\Qp{\eFirstPow}{v}$, there are at least $\eFirstPow k$ numbers $i \in[1,k-1]$  such that $[v+i,v+i+j-1]$ is negative.  Therefore, there is some $i\in[1,k-1]$ such that all the edges in $E(v,i,j)$  are negative and also $[v+i,v+i+j-1]$ is negative.  Part (iv) of Lemma~\ref{nonnegativeintervals} implies that we have $f(v,v+4k+j-1)<0$ and so $(\Qp{\eFirstPow}{v}\setminus T+4k+v)\cap G_{\Hyper{H}}\subseteq R(v)$.   Similarly, using part (v) of Lemma~\ref{nonnegativeintervals}, it is possible to show that $(\Qm{\eFirstPow}{v}\setminus T+4k+v)\cap G_{\Hyper{H}}\subseteq R(v)$. Finally, part (iii) of Lemma~\ref{nonnegativeintervals} implies that we have $(\{0\}+4k+v)\cap G_{\Hyper{H}}\subseteq R(v)$, and hence $Q'+4k+v\subseteq R(v)$.

Now we prove $|\Qe{\eFirstPow}{v}|-\epsilon_2k$.
Since $|T| \leq \ILF{\epsilon_0 C_1 k}{\eFirstPow}$, we must have
\begin{equation}\label{eq:fp1}
|\Qe{\eFirstPow}{v}\setminus T| \geq |\Qe{\eFirstPow}{v}| -\frac{\epsilon_0 C_1}{\eFirstPow} k.
\end{equation}
Condition (iii) of Claim~\ref{transverse} implies that  
\begin{equation} \label{eq:fp2}
|Q'|\geq |\Qe{\eFirstPow}{v}\setminus T|- \epsilon_1 k.
\end{equation}
Now, (\ref{eq:fp1}), (\ref{eq:fp2}) and $\epsilon_2\geq  \ILF{\epsilon_0 C_1}{\eFirstPow} + \epsilon_1$ imply $|Q'|\geq |\Qe{\eFirstPow}{v}|-\epsilon_2k$, proving the claim.
\end{proof}

\begin{definition}
For $S\subseteq A\times B$ we define $$A +_S B = \{a+b : (a,b)\in S\}.$$
\end{definition}
The following claim shows that for a certain large set  $S$,  a translate of $\Qe{\eFirstPow}{0} +_S \Qe{2\eFirstPow}{7k}$ is contained in $R(0)$.

\begin{claim}\label{Qsumset}
There is a set $S\subseteq \Qe{\eFirstPow}{0} \times \Qe{2\eFirstPow}{7k}$ such that $|S|\geq  |\Qe{\eFirstPow}{0}\times \Qe{2\eFirstPow}{7k}|-\eQSum k^2$ and we have 
$$\big(\Qe{\eFirstPow}{0} +_S \Qe{2\eFirstPow}{7k}\big) + 13k \subseteq R(0).$$
\end{claim}
\begin{proof}

For every good vertex $v\in I$, Claim~\ref{firstpower} combined with part (ii) of Claim~\ref{basicR} implies that there is a set $Q_v\subseteq \Qe{\eFirstPow}{v}$ such that  we have  $Q_v+v+\{6k, 7k\} \subseteq R(v)$ and also 
\begin{equation}\label{QvSize}
|Q_v|\geq |\Qe{\eFirstPow}{v}|-\efour k.
\end{equation}

Now, part (i) of Claim~\ref{basicR} implies that we have
\begin{equation}\label{eq:sumset1}
\bigcup_{v\in R(0)\cap [6k, 8k-1] }R(v) \subseteq R(0).
\end{equation}

Combining $Q_v+v+\{6k,7k\} \subseteq R(v)$ with
(\ref{eq:sumset1}) implies that we have
\begin{equation}\label{eq:sumset3}
\bigcup_{v\in (Q_0+\{6k,7k\})}(Q_v +v + \{6k, 7k\})\subseteq R(0).
\end{equation}

We let 
$$S=\{(a,b)\in \Qe{\eFirstPow}{0}\times \Qe{2\eFirstPow}{7k} : a\in Q_0 \text{ and } b\in Q_{a+6k}\cup Q_{a+7k}\}.$$

The identity (\ref{eq:sumset3}) implies that we have 

\begin{align*}
 \Qe{\eFirstPow}{0} +_S \Qe{2\eFirstPow}{7k} +13k &=\{a+b:a\in Q_0 \text{ and }\\
 & \ \ \ \ \ \ \ \ \ \ \ \ \ \ \ \ \  b\in (Q_{a+6k}\cup  Q_{a+7k})\cap \Qe{2\eFirstPow}{7k}\}+13k\\
 &\subseteq \{a+b:a\in Q_0 \text{ and } b\in Q_{a+6k}\cup Q_{a+7k}\}+13k \\
 &=\left(\bigcup_{a\in Q_{0}+6k} Q_a+a+7k \right)\cup \left(\bigcup_{a\in Q_{0}+7k} Q_a+a+6k \right) \\
 &\subseteq \bigcup_{a\in (Q_0+\{6k,7k\})}(Q_a +a + \{6k, 7k\})\\
 &\subseteq R(0).
\end{align*}

Now we prove $|S|\geq  |\Qe{\eFirstPow}{0}\times \Qe{2\eFirstPow}{7k}|-\eQSum k^2$.
Notice that for each $a \in [0,  k-1]$,  part (i) of Claim~\ref{basicQ} implies 
\begin{equation} \label{eqQdoubling}
\Qe{2\eFirstPow}{7k}\subseteq \Qe{\eFirstPow}{a+6k}\cup \Qe{\eFirstPow}{a+7k} \text{ for all } a\in \Qe{\eFirstPow}{0}.
\end{equation}

The identity (\ref{eqQdoubling}) combined with (\ref{QvSize}) and $Q_v\subseteq \Qe{\eFirstPow}{v}$ implies that for all $a\in[1,k-1]$ we have
\begin{align*}
|(Q_{a+6k}\cup Q_{a+7k})\cap \Qe{2\eFirstPow}{7k}|&\geq |(\Qe{\eFirstPow}{a+6k}\cup \Qe{\eFirstPow}{a+7k}) \cap \Qe{2\eFirstPow}{7k}| \\
& \hspace{6.5cm} - 14\epsilon_1+2\epsilon_2) k \\
&= |\Qe{2\eFirstPow}{7k}| -(14\epsilon_1+2\epsilon_2) k.
\end{align*}

This gives us 
\begin{align*}
|S| &= \sum_{a\in Q_0} | (Q_{a+6k}\cup Q_{a+7k}) \cap \Qe{2\eFirstPow}{v}| \\
    &\geq \sum_{a\in Q_0} \Big(|\Qe{2\eFirstPow}{7k}| - (14\epsilon_1+2\epsilon_2) k\Big) \\
    &\geq \Big(|\Qe{\eFirstPow}{0}|-\efour k\Big)\Big(|\Qe{2\eFirstPow}{7k}| - (14\epsilon_1+2\epsilon_2) k\Big) \\
    &\geq  |\Qe{\eFirstPow}{0}\times \Qe{2\eFirstPow}{7k}|- (21\epsilon_1+3\epsilon_2)k^2  \\
    &\geq |\Qe{\eFirstPow}{0}\times \Qe{2\eFirstPow}{7k}|- \eQSum k^2.
\end{align*}
The second last inequality follows from $|\Qe{\eFirstPow}{0}|$, $|\Qe{2\eFirstPow}{7k}|\leq k$.
The last inequality follows from $\eQSum \geq 21\epsilon_1+3\epsilon_2$.
\end{proof}

Claim~\ref{Qsumset} is combined with the following.

\begin{claim}\label{sumsetgrowth}
Suppose that $A$ and $B \subseteq \mathbb{Z}_k$, and satisfy that for any $x\in \mathbb{Z}_k$ , either $x$ or $-x\in A$ and either $x$ or $-x \in B$.  Let $S\subseteq A\times B$ be a set satisfying $|S|\geq |A\times B| - \ethree^2k^2$.  We have 
$$|A+_S B|\geq \left(\frac{1}{2}+\esix \right) k .$$
\end{claim}

When $k$ is prime, Claim~\ref{sumsetgrowth} follows from a theorem due to Lev~\cite{Lev}, which itself is closely related to a theorem due to Pollard~\cite{Pollard}.
In order to prove Claim~\ref{sumsetgrowth}, we will need some results from additive combinatorics.
We define $$(A+B)_i= \{x\in \mathbb{Z}_k: x=a+b \text{ for at least } i \text{ distinct pairs } (a, b)\in A\times B\}.$$  Notice that we have  $(A+B)_{i+1}\subseteq (A+B)_{i}$.

The proof of Claim~\ref{sumsetgrowth} will use the following theorem due to Grynkiewicz.

\begin{theorem}[Grynkiewicz, \cite{Grynk}] \label{gryn}
Let  $A$ and $B \subseteq \mathbb{Z}_k$ and $t\leq k$.  We have one of the following.
\begin{enumerate} [\normalfont(i)]
\item  The following holds.
\begin{equation}\label{eq:gryn1}
\sum_{i=1} ^t |(A+B)_i| \geq t|A| + t|B| - 2t^2 + 1.
\end{equation}
\item There are sets $A'\subseteq A$ and $B'\subseteq B$ such that $|A\setminus A'| + |B \setminus B'| \leq t-1$ and we have $ A' + B' = (A+B)_t$.
\end{enumerate} 
\end{theorem}

We define the \emph{stabiliser} of a set $X\in \mathbb{Z}_k$ to be $Stab(X)=\{y\in \mathbb{Z}_k: y+X=X\}$.  
We use the following theorem due to Kneser.

\begin{theorem} [Kneser, \cite{Kneser}] \label{kneser}
Let $A$ and $B \subseteq \mathbb{Z}_k$ and $H$ the stabiliser of $A+B$ in~$\mathbb{Z}_k$.  We have
\begin{equation} \label{eq:kneser} 
|A+B|\geq |A+ H| + |B+H| -|H|. 
\end{equation}
\end{theorem}

Sumsets in Claim~\ref{sumsetgrowth}, Theorem~\ref{gryn} and Theorem~\ref{kneser} are all in $\mathbb{Z}_k$.

\begin{proof} [Proof of Claim~\ref{sumsetgrowth}]

Notice that since $x$ or $-x\in A, B$, we must have $|A|, |B|\geq \frac{1}{2}k$ .
Our initial goal will be to show that we have 
\begin{equation}\label{eqABt}
|(A+B)_{\ethree k}|\geq  \left(\frac{1}{2}+\esix + \ethree\right)k.
\end{equation}

Apply Theorem~\ref{gryn} to $A$ and $B$ with $t=2 \ethree k$.  We split into two cases, depending on which part of Theorem~\ref{gryn} holds.

\begin{enumerate} [\normalfont(i)]
 \item  
  Suppose that (\ref{eq:gryn1}) holds.  Since we are working over $\mathbb{Z}_k$ in this claim, we have $|(A+B)_i|\leq k$.  Combining this with (\ref{eq:gryn1}) implies
\begin{align*}
 \sum_{i=\ethree k} ^{2\ethree k} |(A+B)_i| &\geq 2\ethree k\Big(|A| +  |B|-4\ethree k\Big)+1- \sum_{i=1} ^{\ethree k-1} |(A+B)_i|\\
 &\geq \ethree k\Big(2|A| +  2|B|-(1+8\ethree)k\Big).
\end{align*}
This, together with $(A+B)_{i+1}\subseteq (A+B)_{i}$ implies that we have
\begin{equation*}
 |(A+B)_{\ethree k}|\geq 2|A|+2|B|-(1+8\ethree)k.
\end{equation*}
The identity (\ref{eqABt})  follows since we have $|A|, |B|\geq \frac{1}{2}k$  and $1-8\ethree \geq \ILF{1}{2}+\esix + \ethree$.

\item Suppose that we have two sets $A'$ and $B'$ as in  part (ii) of Theorem~\ref{gryn}.  
Apply Theorem~\ref{kneser} to the sets $A'$ and $B'$.  

Note that $|A\setminus A'| + |B \setminus B'| \leq t-1$ together with (\ref{eq:kneser}) and  $|A|, |B|\geq \frac{1}{2}k$ implies that we have 
\begin{align}
|(A+B)_{\ethree k}| &\geq |(A+B)_{2\ethree k}| \notag\\
&=|A' + B'| \notag\\
&\geq |A'+Stab(A'+B')|+|B'+Stab(A'+B')|-|Stab(A'+B')|\label{eqSymBound1}\\
&\geq |A|+|B|-|Stab(A'+B')|-2\ethree k \notag\\
&\geq (1-2\ethree)k-|Stab(A'+B')|. \label{eqSymBound2}
\end{align}
If $|Stab(A'+B')|\leq \frac{1}{3}k$, then (\ref{eqABt}) follows (\ref{eqSymBound2}) combined with $1-2\ethree-\ILF{1}{3}\geq \ILF{1}{2}+\esix + \ethree$.

Otherwise, Lagrange's Theorem implies that $Stab(A'+B')$ is either all of $\mathbb{Z}_k$ or that~$k$ is even and $Stab(A'+B')$ is the set of even elements of $\mathbb{Z}_k$.  
If  $Stab(A'+B') = \mathbb{Z}_k$ holds, then we have $A'+Stab(A'+B')=B'+Stab(A'+B')= \mathbb{Z}_k$.  Substituting this into (\ref{eqSymBound1}) implies that we have $|(A+B)_{\ethree k}|=k$ and so (\ref{eqABt}) holds.  

Suppose that $Stab(A'+B')$ consists of all the even elements of $\mathbb{Z}_k$.  Since for every~$x$, either $x$ or $-x \in A$, there are at least $\frac{1}{4}k$ even elements in $A$, and at least $\frac{1}{4}k$ odd elements in $A$.  Therefore, since $|A'|\geq |A|-2\ethree k$, $A'$ must contain an even element and an odd element.  This implies that $A'+Stab(A'+B')=\mathbb{Z}_k$. Similarly $B'+Stab(A'+B')=\mathbb{Z}_k$.  Thus  (\ref{eqSymBound1}) implies that  we have $|(A+B)_{\ethree k}|=k$ and so (\ref{eqABt}) holds.    
\end{enumerate}

Now, we use (\ref{eqABt}) to deduce the claim.
Let $T=(A+B)_{\epsilon_3 k}\setminus (A+_S B)$.  We have $|A+_S B|+|T|\geq |(A+B)_{\ethree k}|$.
Notice that from the definition of $(A+B)_{\ethree k}$ we have $\ethree k |T|+|S|\leq |A\times B|$.
  This, combined with (\ref{eqABt}) and $|S|\geq |A\times B| - \ethree^2 k^2$ implies that we have  
  \begin{align*}
  |A+_S B|&\geq |(A+B)_{\ethree k}|-|T| \\
  &\geq |(A+B)_{\ethree k}|-\frac{1}{\ethree k}(|A\times B|-|S|)\\
  &\geq  |(A+B)_{\ethree k}|-\ethree k\\
  &\geq  \left(\frac{1}{2} + \esix\right)k.
 \end{align*}
 \end{proof}

Claims~\ref{Qsumset} and \ref{sumsetgrowth} cannot be directly combined since sumsets in Claim~\ref{Qsumset} are in $\mathbb{Z}$ whereas sumsets in Claim~\ref{sumsetgrowth} are in $\mathbb{Z}_k$.
However, Claim~\ref{Qsumset} gives us a set $S$ such that $|S|\geq  |\Qe{\eFirstPow}{0}\times \Qe{2\eFirstPow}{7k}|-\eQSum k^2$ and we have 
$\big(\Qe{\eFirstPow}{0} +_S \Qe{2\eFirstPow}{7k}\big) + 13k \subseteq R(0).$  
Part (iii) of Claim~\ref{basicR} implies that there is a subset $Q'\subseteq (\Qe{\eFirstPow}{0} +_S \Qe{2\eFirstPow}{7k})\bmod(k)$ such that $Q'+16k\subseteq R(0)$ and we have
\begin{equation}\label{eqm1}
|Q'|\geq |(\Qe{\eFirstPow}{0} +_S \Qe{2\eFirstPow}{7k})\bmod(k)|-3\epsilon_1k.  
\end{equation} 
By Claim~\ref{sumsetgrowth} and part (ii) of Claim~\ref{basicQ}, we have 
\begin{equation}\label{eqm2}
|(\Qe{\eFirstPow}{0} +_S \Qe{2\eFirstPow}{7k})\bmod(k)|\geq  \left(\frac{1}{2}+ \esix\right)k.
\end{equation} 
Combining (\ref{eqm1}) and (\ref{eqm2}) implies that $|R(0)\cap [16k, 17k-1]|\geq \left(\ILF{1}{2}+ \esix-3\right)\epsilon_1k$.  Applying part (iv) of Claim~\ref{basicR} implies that for any $w\in I$, we have
\begin{equation*}
|R(0)\cap [w,w+k-1]|\geq \bigg(\frac{1}{2}+ \esix- \epsilon_1\left(\frac{w}{k}+4\right)\bigg) k. 
\end{equation*}
Combining this with $m\leq (4C_2+C_3)k$ gives
\begin{align}\label{eq0}
|R(0)\cap [m-17k,m-16k-1]|\geq \left(\frac{1}{2}+\esix -\epsilon_1(4C_2+C_3+4)\right).
\end{align}
We can define $R^-(v)= \{u \in I \cap G_{\Hyper{H}} : f([u+1, v]) < 0 \}$.  By symmetry, we obtain 
\begin{equation}\label{eqm}
|R^-(m)\cap [m-17k,m-16k-1]|\geq \left(\frac{1}{2}+\esix-3\epsilon_1\right)k.
\end{equation}
Now, (\ref{eq0}), (\ref{eqm}), and $\esix>\epsilon_1(4C_2+C_3+4)$ imply that we have 
\begin{align*}
|R(m)\cap [m-17k,m-16k-1]|&>\frac{1}{2}k, \\
|R^-(m)\cap [m-17k,m-16k-1]|&>\frac{1}{2}k.
\end{align*}
Therefore, there is some $v\in [m-17k,m-16k-1]$ such that $v\in R(0)$ and $v-1\in R^-(m)$.  By definition of $R(0)$ and $R(m)$ we obtain $f(I)<0$.
\end{proof}
As mentioned before, Claims \ref{cover} and \ref{transverse} imply the theorem.
\end{proof}

%% file: MMSDiscussion.tex
In this section we discuss some further directions one might take with our approach to Conjecture~\ref{MMSConjecture}.

\begin{itemize}
\item
The constant $10^{46}$ in Theorem~\ref{MMSMainTheorem} can  certainly be improved by being more careful in the proof.  The main question is whether a better choice of hypergraphs $\Hyper{H}_{n,k}$ can lead to a solution to Conjecture~\ref{MMSConjecture}.
It is not clear what kind of hypergraphs one should look for.  Although in the above theorem, the hypergraphs $\Hyper{H}_{n,k}$ are quite sparse, this does not seem to be crucial in the proof.  

\item 
The constant ``$10^{46}$'' cannot be reduced to ``$4$'' in Theorem~\ref{MMSMainTheorem} without changing the graphs $\Hyper{H}_{n,k}$.  Indeed for large $k$, the graphs $\Hyper{H}_{5(k-1), k}$ do not have the MMS-property.  To see this, consider the following function $f:V(G)\to \mathbb{R}$.
\begin{align*}
f(i)&=k-2 \text{ if } i\equiv 0 \pmod{k-1},\\
f(i)&=-1 \text{ if } i\not\equiv 0 \pmod{k-1}.
\end{align*}
It is easy to see that we have $\sum_{x \in V(G)} f(x)=0$.   For two vertices $i$ and $j$ let
$$p(i,j)=\begin{cases}\text{The number of edges of } \Hyper{H}_{5(k-1), k} \text{ containing }i\text{ and }j &\text{ if }i\neq j \\
          0 &\text{ if } i=j.
         \end{cases}$$
The graph $\Hyper{H}_{5(k-1), k}$ has five nonnegative vertices $0, k-1,2(k-1), 3(k-1), 4(k-1)$.  An edge $e\in \Hyper{H}_{5(k-1), k}$ is nonnegative if and only if $e$ contains at least two of these vertices.  Therefore the number of nonnegative edges in $\Hyper{H}_{5(k-1), k}$ is at most
\begin{equation}
\frac{1}{2}\sum_{\substack{i,j\in\{0, k-1,2(k-1),\\ \ \ \ \ \ \ 3(k-1), 4(k-1)\}}} p(i,j)=5p(0,k-1)+ 5p(0,2(k-1)).
\end{equation}
Notice that an edge $e(-v,i,j)$ contains both $0$ and $k-1$ if and only if we have
\begin{align}
i&\geq v+1, \label{vijdisc1}\\
j&\geq v, \label{vijdisc2}\\
i+j&\geq v+k-1.\label{vijdisc3} 
\end{align}
It's easy to check that the number of triples $(v,i,j)$ which satisfy (\ref{vijdisc1}) -- (\ref{vijdisc3}) is less than $\frac{1}{6}k^3+o(k^3)$, which implies that $p(0,k-1)=\frac{1}{6}k^3+o(k^3)$.

The only edges $\Hyper{H}_{5(k-1), k}$ which contain $0$ and $2(k-1)$ are of the form $e(0,i,k-1)$ for some $i$, so we have that $p(0,2(k-1))=k-1$.  Therefore, there are less than $\frac{5}{6}k^3+o(k^3)$ nonnegative edges in $\Hyper{H}_{5(k-1), k}$ which is smaller than $k(k-1)^2$ for large enough $k$.

The above argument shows that the constant ``$10^{46}$'' in Theorem~\ref{MMSMainTheorem}  cannot be reduced to less than $5$.  This shows that  Conjecture~\ref{MMSConjecture} cannot be solved by the argument we used in this \body without changing the graphs  $\Hyper{H}_{n,k}$ to some other construction.

\item
We conclude with the following general problem.
\begin{problem}\label{MMSProperty}
Which hypergraphs have the MMS-property?
\end{problem}
This problem is probably quite hard, since a solution to it would mean a generalization of Conjecture~\ref{MMSConjecture}.  However, perhaps looking for hypergraphs which have the MMS-property would lead to improved bounds on Conjecture~\ref{MMSConjecture}.

\end{itemize}

%% file: MMSLinearNov.bbl
\begin{thebibliography}{10}

\bibitem{Sudakov}
N.~Alon, H.~Huang, and B.~Sudakov.
\newblock Nonnegative k-sums, fractional covers, and probability of small
  deviations.
\newblock {\em J. Combinatorial Theory Ser. B}, 102:784--796, 2012.

\bibitem{CombinatoricsBook}
B.~Bollob\'as.
\newblock {\em Combinatorics}.
\newblock Cambridge University Press, 1986.

\bibitem{ErdosKoRado}
P.~Erd\H{o}s, C.~Ko, and R.~Rado.
\newblock Intersection theorem for system of finite sets.
\newblock {\em Quart. J. Math. Oxford Ser.}, 12:313--318, 1961.

\bibitem{Marino}
G.~C. G.~Marino.
\newblock A method to count the positive 3-subsets in a set of real numbers
  with non-negative sum.
\newblock {\em European J. Combin.}, 23:619--629, 2002.

\bibitem{Grynk}
D.~J. Grynkiewicz.
\newblock On extending {P}ollard's theorem for t-representable sums.
\newblock {\em Israel J. Math}, 177(1):413--439, 2010.

\bibitem{Katona}
G.~O.~H. Katona.
\newblock A simple proof of the {E}rd{\H o}s-{C}hao {K}o-{R}ado {T}heorem.
\newblock {\em J. Combin. Theory Ser. B}, 13:183--184, 1972.

\bibitem{Kneser}
M.~Kneser.
\newblock Absch\"atzung der asymptotischen {D}ichte von {S}ummenmengen.
\newblock {\em Math. Z}, 58:459--484, 1953.

\bibitem{Lev}
V.~F. Lev.
\newblock Restricted set addition in groups, ii. a generalization of the
  {E}rd{o}s-{H}eilbronn conjecture.
\newblock {\em Electron. J. Combin.}, 7, 2000.

\bibitem{ManickamSmalln}
N.~Manickam.
\newblock {\em On the distribution invariants of association schemes}.
\newblock PhD thesis, Ohio State University, 1986.

\bibitem{Manickam}
N.~Manickam and D.~Mikl\'os.
\newblock On the number of non-negative partial sums of a non-negative sum.
\newblock {\em Colloq. Math. Soc. J\'anos Bolyai}, 52:385--392, 1987.

\bibitem{Singhi}
N.~Manickam and N.~Singhi.
\newblock First distribution invariants and {EKR} theorems.
\newblock {\em J. Combin. Theory Ser. A}, 48:91--103, 1988.

\bibitem{Pollard}
J.~M. Pollard.
\newblock A generalisation of the {T}heorem of {C}auchy and {D}avenport.
\newblock {\em J. London Math. Soc.}, 8(2):460--462, 1974.

\bibitem{Tyomkyn}
M.~Tyomkyn.
\newblock An improved bound for the {M}anickam-{M}ikl\'os-{S}inghi conjecture.
\newblock {\em European J. Combin.}, 33(1):27--32, 2012.

\end{thebibliography}
